\documentclass[11pt]{amsart}
\usepackage[inline]{enumitem}
\usepackage{hyperref,amssymb,verbatim,graphics,epsfig,psfrag,float}
\usepackage{caption}
\usepackage{subcaption}

\hypersetup{colorlinks}
\usepackage{color}
\definecolor{darkred}{rgb}{0.5,0,0}
\definecolor{darkgreen}{rgb}{0,0.5,0}
\definecolor{darkblue}{rgb}{0,0,0.5}
\hypersetup{ colorlinks,
linkcolor=darkblue,
filecolor=darkgreen,
urlcolor=darkred,
citecolor=darkblue }

\usepackage{amssymb,verbatim,graphicx}
\newtheorem{theorem}{Theorem}
\newtheorem{btheorem}{$\beta$-Theorem}[section]

\newtheorem{conjecture}[btheorem]{Conjecture}
\newtheorem{proposition}[btheorem]{Proposition}
\newtheorem{lemma}[btheorem]{Lemma}

\theoremstyle{definition}
\newtheorem{definition}[btheorem]{Definition}
\theoremstyle{remark}
\newtheorem{remark}[btheorem]{Remark}

\newcommand\mU{\mathcal{U}}

\newcommand{\V}{\mathcal{V}}
\renewcommand{\L}{\mathcal{L}}

\newcommand{\F}{\mathcal{F}}

\newcommand{\R}{\mathbb{R}}

\newcommand{\C}{\mathbb{C}}

\newcommand{\Z}{\mathbb{Z}}

\newcommand{\on}{\operatorname}

\renewcommand{\ker}{ \on{ker}}

\newcommand\Id{\on{Id}}

\newcommand\bs{\backslash}

\newcommand\lam{\lambda}

\newcommand\eps{\epsilon}
\newcommand\Om{\Omega}
\newcommand\om{\omega}

\newcommand{\thh}{{\tfrac{1}{2}}}

\newcommand\Mod[1]{\lVert #1 \rVert}
\newcommand\qu{/\kern-.7ex/} 

\newcommand\bS{\mathbb{S}}

\newcommand\bD{\mathbb{D}}
\newcommand\bT{\mathbb{T}}

\begin{document}
\title{Novikov's theorem in higher dimensions?}

\author{Sushmita Venugopalan} \address{Institute of Mathematical
  Sciences, CIT Campus, Taramani, Chennai 600113, India.}
\email{sushmita@imsc.res.in}

\begin{abstract}
  Novikov's theorem is a rigidity result on the class of taut
  foliations on three-manifolds.  For higher dimensional manifolds,
  foliations with a strong symplectic form have been suggested as the
  class of foliations having similar rigidity properties to taut
  foliations on three-manifolds.  This leads to the natural question
  of whether strong symplectic foliations satisfy an analogue of
  Novikov's theorem.  In this paper, we construct a five-dimensional
  manifold with a smooth foliation and a strong symplectic form that
  does not satisfy the expected analogue of Novikov's theorem. Our
  example is a foliated Lefschetz fibration.
\end{abstract}

\maketitle


\section{Introduction}
\subsection{Novikov's theorem for three-manifolds}
Novikov's theorem is a result on foliated $3$-manifolds that do not
contain Reeb components.  A {\em Reeb component} in a foliated
$3$-manifold $X$ is a solid torus $\bD^2 \times \bS^1 \subset X$
equipped with the Reeb foliation; the {\em Reeb foliation} on a solid
torus $\bD^2 \times \bS^1$ (\cite{reeb}, \cite[Example 1.1.12]{CC1})
has the boundary $\bS^1 \times \bS^1$ as a leaf, and all other leaves
are homeomorphic to planes. The ends of the non-compact leaves wind
around the torus, progressively getting closer to the boundary leaf,
see Figure \ref{fig:torus-reeb}.


  \begin{figure}[h]
    \centering \scalebox{.7}{ 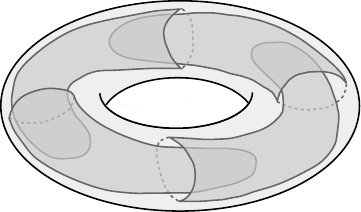}
    \caption{Reeb foliation on the solid torus.}
    \label{fig:torus-reeb}
  \end{figure}
\begin{proposition} \label{prop:Novikov} {\rm(Novikov's theorem
    \cite{Nov})} Suppose $X$ is a compact $3$-manifold and $\F$ is a
  codimension one foliation that does not contain any Reeb component.
  Then,
  \begin{enumerate}
  \item for any leaf $\L$ of $\F$, the map $\pi_1(\L) \to \pi_1(X)$ is
    injective,
  \item for any loop $\gamma:\bS^1 \to X$ transverse to $\F$, the
    homotopy class $[\gamma]$ is non-trivial in $\pi_1(X)$, and
  \item $\pi_1(X)$ is infinite.
  \end{enumerate}
\end{proposition}

\noindent We refer the reader to \cite[Theorem 24]{Lawson},
\cite[Section 9.1]{CC2} for expositions of this result.  Novikov's
theorem is a {\em rigidity} result on Reebless foliations in the sense
that it imposes restrictions on the topology of the underlying
manifold.  On the other hand, if Reeb components are allowed, a
foliation exists on a manifold if its Euler characteristic
vanishes. Furthermore, by Thurston's result \cite{Thu}, any
distribution $\xi_0 \subset TX$ of codimension one in a three manifold
$X$ can be homotoped to an integrable distribution, and we recall that
an integrable distribution is the leaf tangent space of a foliation.

A codimension one foliation on a manifold (of any dimension) is {\em
  taut} if through every point of the manifold, there is a loop that
is transverse to the leaves of the foliation. On a $3$-manifold, a
taut foliation is Reebless, because a transversal that enters a Reeb
component cannot exit it. In the absence of an analogue of a Reeb component in higher
dimensions, one may attempt to generalize Novikov's theorem after
replacing the ``Reebless'' property by ``tautness''.

  \subsection{Rigidity in higher dimensions}

  In higher dimensional manifolds, taut foliations do not have
  rigidity properties similar to the $3$-dimensional case. They do not
  satisfy a Novikov-like theorem, and furthermore, Meigniez \cite{Mei}
  showed that on a manifold of dimension greater than $3$, any
  distribution of codimension one is homotopic to the leaf tangent
  bundle of a taut foliation.


  In higher
  dimensional manifolds, 
  Martinez-Torres \cite{MT2} proposed {\em strong symplectic
    foliations} as the analogue for taut foliations on three-manifolds. On a $(2n+1)$-dimensional manifold $X$ with a
  codimension one foliation $\F$, a {\em strong symplectic} form
  $\om \in \Om^2(X)$ is a closed two-form that restricts to a
  symplectic form on the leaves. The triple $(X,\F,\om)$ is called a
  {\em strong symplectic foliation}.  The motivation in \cite{MT2} to
  consider strong symplectic foliations is as follows: By Sullivan
  \cite{Sul}, a foliated $n$-manifold is taut if and only if there is
  a closed $(n-1)$-form on the manifold that induces a volume form on
  each leaf of the foliation; \footnote{See Section 10.4 in
    Candel-Conlon \cite{CC1}.}  and results on volume geometry in two-manifolds,
  such as the Poincar\'e-Birkhoff theorem, naturally generalize to
  symplectic manifolds in higher dimensions (\cite[page 339]{ms:small}).
  Furthermore, techniques of rigid symplectic
  geometry, such as Donaldson divisors and pseudoholomorphic curves, 
  have been fruitfully applied to strong symplectic foliations
  \cite{IMT2, imt:approx, MT1, MT2, MPP, PV:fill}.

  Martinez-Torres' suggestion leads to the question of whether strong
  symplectic foliations satisfy an analogue of Novikov's theorem. The
  answer to this question turns out to be negative, and in this paper,
  we present a counter-example.

\begin{theorem}\label{thm:main}
  There is a compact manifold $X^5$ with a smooth foliation $\F^4$ and a   
  strong symplectic form $\om$
  for which
  \begin{enumerate}
  \item \label{part1} the map $\pi_1(\F) \to \pi_1(X)$ is not
    injective, and 
  \item \label{part2} there is a loop which is transverse to the
    foliation $\F$ that is contractible in $X$.
  \end{enumerate}
\end{theorem}

However, it is still likely that strong symplectic foliations are
rigid objects similar to taut foliations on three-manifolds. The example we construct is
not simply connected, and we expect the following weaker Novikov-type
result to hold for strong symplectic foliations.

\begin{conjecture}
  A compact odd-dimensional manifold $X$ with a codimension one foliation $\F$ and a strong symplectic form $\om \in \Om^2(X)$  cannot be simply connected.
\end{conjecture}

Compared to strong symplectic foliations, weak symplectic foliations
are more flexible objects.  A codimension one foliation $(X,\F)$ is
{\em weak symplectic} if there is a two-form that restricts to a
symplectic form on leaves, but is not required to be closed in $X$.
By 
\cite{IMT2}, any finitely presented group can be made to be the
fundamental group of a weak symplectic foliation. Mitsumatsu
\cite{Mit} showed that, in particular, $\bS^5$ admits a weak
symplectic foliation. Mitsumatsu's construction contains a compact
leaf. In contrast, a strong symplectic foliation on a simply connected
manifold cannot have a compact leaf $L$.  Indeed, if a strong
symplectic foliation $(X^{2n+1},\F^{2n} , \om)$ admits a compact leaf,
then $[\om^n]$ is a non-zero class in $H^{2n}(X)$ whose Poincar\'e
dual is a non-trivial element in $H^1(X)$, contradicting the simply
connectedness of $X$.

Mitsumatsu's foliation is not taut, however, Gironella,
Niederkr\"uger, and Toussaint showed in \cite{gnt:vc} (which 
appeared after the first version of this paper) that Mitsumatsu's
foliation can be modified to a taut foliation that is weak symplectic
but not strong symplectic.


More broadly, 
\cite{gnt:vc} proves a rigidity
result for strong symplectic foliations, namely that a strong
symplectic foliation does not contain a non-trivial {\em Lagrangian
  vanishing cycle} which is a submanifold with a singular foliation
whose smooth leaves are Lagrangian.  In $3$-dimensional foliations,
the presence of such an object is equivalent to the presence of a Reeb
component.

We present our example in Section \ref{sec3}. In Section \ref{sec2},
we review results about symplectic Lefschetz fibrations and adapt them
to the foliated setting.

\subsection{Acknowledgements} I am indebted to Fran Presas for
introducing me to the question of whether Novikov's theorem extends to
higher dimensions. I thank both him and \'Alvaro del Pino for
discussions on this topic. I thank the referee for their detailed
comments and suggestions on the paper.  I also acknowledge support of
the grant 612534 MODULI within the 7th European Union Framework
Programme.

\section{Symplectic Lefschetz fibrations}\label{sec2}
The building block of our example is a foliated symplectic Lefschetz
fibration. In this section, we construct a strong symplectic form on a
topological Lefschetz fibration with a foliated base.  This is an
extension of Gompf's construction \cite{Gompf:toward} of symplectic
forms on singular Lefschetz fibrations, which in turn, is an extension
of Thurston's construction (\cite{Thurston:eg}, \cite[Theorem
6.3]{ms:small}) of a symplectic form on fibrations without
singularities.

In our definition of the foliated Lefschetz fibration, fibers are
two-dimensional and tangent to the foliation of the total space. A
variant of this structure has been defined and studied in
\cite[Definition 8]{MT2} and \cite{imt:lef}, wherein the fibers are
transverse to the foliation of the total space, and each fiber is a
foliated three-manifold in the complement of fibration
singularities. In both versions -- ours and that of \cite{MT2} -- each
leaf of the total space is a Lefschetz fibration. To the best of our
knowledge, the version of foliated Lefschetz fibration that we define
has not previously occurred in the literature.

We state the main result, Proposition \ref{prop:symplef}, after
defining the foliated and unfoliated versions of Lefschetz fibrations.

\begin{definition}\label{def:TLT}
A {\em topological Lefschetz fibration} consists of an oriented
four-manifold $X$, and a proper map $\pi:X \to B$ to an oriented
surface $B$ satisfying the following property: For any $x_0 \in X$,
\begin{enumerate}
\item either $d\pi_{x_0}$ is surjective, or
\item there are complex coordinates
$(z_1,z_2)$ on a neighborhood $U_{x_0} \subset X$ of $x_0$ that respect
orientation and satisfy $(z_1,z_2)(x_0)=(0,0)$,
 and a complex coordinate $w$ on a neighborhood $U_{\pi(x_0)} \subset B$ such that
 \[(w \circ \pi)=z_1^2+z_2^2\]
 on $U_{x_0}$. In this case $x_0$ is a {\em singular point} and $\pi(x_0)$ is a singular value. The set of singular points is denoted by $\Delta$.
\end{enumerate}
We require that the projection $\pi$ is injective on the set of its
singular points.  
\end{definition}

The definition implies that, when restricted to $X \bs \Delta$, $\pi$
is a submersion, and is therefore, it is an honest fibration on
$B \bs \pi(\Delta)$ whose fiber is an oriented surface.  We denote the
regular fiber by $F$, and for any $b \in B$, we denote the fiber
$\pi^{-1}(b)$ by $F_b$. Note that $\pi:X \to B$ is actually a singular
fibration, but the word `singular' is typically dropped in the
literature on Lefschetz fibrations.

\begin{remark}\label{rem:nbhd}
  {\rm(Neighborhood of a singular fiber)}
We describe the neighborhood of a singular fiber in a Lefschetz
fibration $\pi : X \to B$.  Let $x_0 \in X$ be a singular point, and
let $U_{b_0} \subset B$ be a neighborhood of $b_0:=\pi(x_0)$ with no
singular values except $b_0$.  After fixing an identification
$F_b \simeq F$ for some $b \in U \bs \{b_0\}$, there is
a loop $V \subset F$ that is pinched to a point in the singular fiber
$F_{b_0}$. The loop $V$ is uniquely defined up to isotopy, and is
called the {\em vanishing cycle} associated to $x_0$.  The monodromy
of the fibration $\pi^{-1}(U_b \bs \{b\})$ around a simple loop
$\gamma \subset U_{b_0} \bs \{b_0\}$ going around $b_0$
counter-clockwise is a $+1$ Dehn twist about $V$ (which we recall is
an element of the mapping class group of $F$). See \cite[page
291]{Gompf_book} for details. Furthermore, there is a family of loops
$V_b \subset F_b$, $b \in U \bs \{b_0\}$, whose complement is a
trivial fibration, that is, there is an orientation-preserving diffeomorphism
  \begin{equation}
    \label{eq:triv}
    \pi^{-1}(U) \bs (\{x_0\} \cup (\cup_b V_b)) \simeq U \times (F \bs V).   
  \end{equation}
\end{remark}

The next result shows that a neighborhood of a singular fiber of a
Lefschetz fibration can be endowed with a symplectic form, which is
non-degenerate on the fiber tangent space at any regular point
$x \in X$, and which integrates to a prescribed value on fibers.

\begin{lemma}\label{lem:singsymp}
{\rm(A  symplectic form near a singular fiber)}
Let $\pi: X \to B$ be a topological Lefschetz fibration whose regular fiber is a compact oriented surface $F$
and let $x_0 \in X$ be a singular point. Let $\lam>0$. 
There is a neighborhood
$U \subset B$ of $b_0:=\pi(x_0)$, and a symplectic form $\om$ on $\pi^{-1}(U)$ that is  
\begin{enumerate}
\item non-degenerate on the fiber tangent spaces $\ker d\pi_x$ for regular points $x \in \pi^{-1}(U)$, and
\item the $\om$-area of a regular fiber of $\pi$ is $\lam$.
\end{enumerate}
\end{lemma}

Our proof is similar to the one in Amoros et al. \cite[Proposition 2.3]{Amoros},
 with some more details added. 
   \begin{figure}[h]
    \centering \scalebox{.7}{ 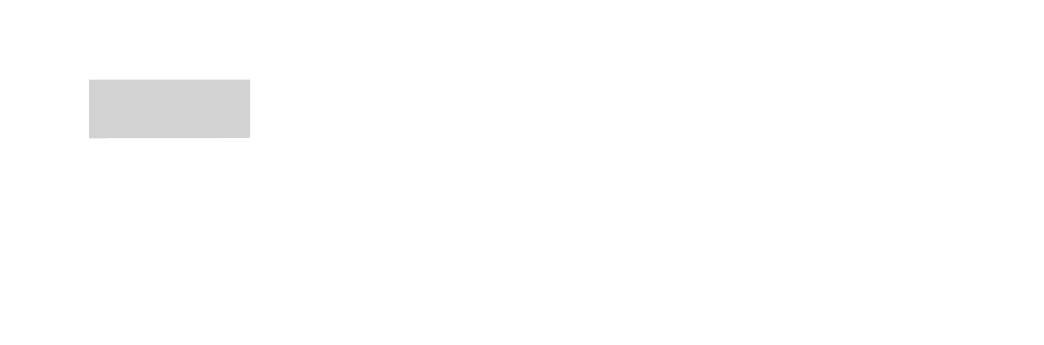}
    \caption{Construction of symplectic form near a singular fiber.}
    \label{fig:singform}
  \end{figure}
  \begin{proof}[Proof of Lemma \ref{lem:singsymp}]
    We will construct the symplectic form by gluing a standard
    symplectic form in a neighborhood of a singular point with a
    product symplectic form on a trivial fibration.
    For a singular point $x_0$ of the fibration, we use the complex coordinates
    $(z_1,z_2)$ given by Definition \ref{def:TLT}, and denote by $\om_0$ the 
    pullback of the standard symplectic form on $\C^2$ in a neighborhood $U_{x_0} \subset X$. 
    For a small $\delta>0$, let $A_0 \subset F_{b_0}$ be a pair of
    annuli, each of $\om_0$-area $\delta$, such that the region in
    $F_{b_0}$ in between the annuli contains the node $x_0$ and has an area of $\delta$.

For a small enough neighborhood $U$ of the singular value $b_0$, 
we will extend the symplectic form to $\pi^{-1}(U)$ by identifying the complement of a neighborhood of the singular point to
a trivial fibration. 
   Consider the trivial fibration $X_{triv}:=U \times F$
   equipped with projection maps and a product symplectic form
   \[\pi_{triv} : X_{triv} \to U, \quad \pi_F : X_{triv} \to F, \quad \om_{triv}=\pi_F^*\om_F + \pi_{triv}^*\om_B,\]
   where $\om_B \in \Om^2(B)$, $\om_F \in \Om^2(F)$ are non-vanishing area forms, and $\int_F\om_F= \lam$. Let
    $A \subset F$ consist of a pair of annuli,
  both homotopic to the vanishing cycle $V$, each having area
  $\delta$, and enclosing an annulus of area $\delta$ between
  them.  See Figure \ref{fig:singform}.  Let
  \[X \supset A_0 \xrightarrow{\phi} \{b_0\} \times A \subset X_{triv}\]
  be an area-preserving map between the two pairs of annuli, that is extendable to 
  an orientation-preserving diffeomorphism between $F_{b_0} \bs \{x_0\}$ and $F \bs V$.   
  Since the normal bundles of $A_0$ and $\{b_0\} \times A$ are both trivial, 
   the symplectic neighborhood theorem implies that $\phi$ extends to a symplectomorphism 
\[\phi: N(A_0) \to N(\{b_0\} \times A)\]
between neighborhoods $N(A_0) \subset X$,
$N(\{b_0\} \times A) \subset X_{triv}$.  Note that $\phi$ is not
fiber-preserving for the maps
$\pi$, $\pi_{triv}$, with the exception that $A_0$ is mapped to
a fiber of $\pi_{triv}$.  
We shrink $U$ so that $N(A_0) \cap F_b$ is a pair of
annuli for each $b \in U$ with a region $R_b \subset F_b$ between the
annuli, and let $R_X \subset \pi^{-1}(U)$ be the union of the
regions $R_b$.
Similarly, we can also ensure that the intersection of $N(A \times \{b_0\})$ and any fiber $F \times \{b\}$, $b \in U$ of $\pi_{triv}$ is a pair of annuli, and the union of all the pairs of annuli is denoted by $R_{triv} \subset X_{triv}$.
We extend $\phi$ to a diffeomorphism (see \eqref{eq:triv}) 
\[\phi : \pi^{-1}(U) \bs R_X \to  \pi_{triv}^{-1}(U) \bs R_{triv} ,\]
such that the fiber $F_{b_0}$ is mapped to $\{b_0\} \times F$. 
Since $\phi^*\om_{triv}=\om_0$ on $N(A_0)$, 
\[\om:=
  \begin{cases}
    \om_0, \quad  \text{on } (R_X \cup N(A_0))\\
    \phi^* \om_{triv}, \quad \text{ on } \pi^{-1}(U) \bs (R_X \cup N(A_0)
    \end{cases}\]
  is a smooth symplectic form on $\pi^{-1}(U)$. Since $\om$ is
  non-degenerate on $\ker(d\pi)$ on the regular points of
  $R_X \cup F_{b_0}$, after possibly shrinking $U$, $\om$ is
  non-degenerate on $\ker(d\pi)$ on the regular points of all of
  $\pi^{-1}(U)$.  Since $\om$ is closed, the area of all the smooth
  fibers of $\pi$ are equal, and in the limit, the area of the
  singular fiber is also the same. By our construction, the $\om$-area
  of the singular fiber is $\lam$. This finishes the proof of the
  proposition.
\end{proof}

\begin{definition}\label{def:folTLT}
A {\em topological Lefschetz fibration on a foliated base manifold}
consists of a foliated five-manifold $(X,\F_X)$ and a proper map to a foliated three
manifold $(B,\F_B)$
\begin{equation*}
  \pi:(X,\F_X) \to (B,\F_B)
\end{equation*}
that is a map of foliations in the sense that $\F_X:=\pi^*\F_B$ and which satisfies 
 the following: For any $x_0 \in X$,
\begin{enumerate}
\item either $d\pi_{x_0}$ is surjective, or
\item there are coordinates $(z_1,z_2,t) \in \C^2 \times \R$ in a neighborhood $U_{x_0} \subset X$ of $x_0$ and coordinates $(w,\tau) \in \C \times \R$ on a neighborhood $U_{\pi(x_0)} \subset B$ of $\pi(x_0)$ such that $t$, $\tau$ are constant on leaves of $\F_X$, $\F_B$ respectively, and
  \[((w,\tau) \circ \pi)=(z_1^2 + z_2^2, t) \]
  on $U_{x_0}$. The point $x_0$ is a {\em singular point} and $\pi(x_0)$ is a {\em singular value}, and  the restriction of $\pi$ to the set of singular points is injective. 
\end{enumerate}
The manifolds $X$, $B$ and the foliations $\F_X$, $\F_B$ are assumed to be oriented. 
We just call these objects ``topological Lefschetz fibrations'' when it
is clear from the context that the base and total space are foliated
manifolds. This ends the definition.
\end{definition}

The definition implies that on each leaf $L$ of $X$, $\pi|L$ is a Lefschetz
fibration over the leaf $\pi(L)$ in $B$. The map $\pi$ is submersive
on leaf spaces,
\footnote{Strictly speaking, we mean here ``local leaf spaces''. Recall from \cite[p22]{Cal}
that in a chart $U \times (-\eps,\eps) \subset (X,\F_X)$ with leaves $U \times \{point\}$, the local leaf space is $(-\eps,\eps)$, and is given by mapping each leaf in the chart to a point.
}
that is, for any $x \in X$, the induced map
$d\pi_x: T_xX/T_x\F_X \to T_{\pi(x)}B/T_{\pi(x)}\F_B$ is surjective.
Furthermore, in our examples, $X$ is compact, and therefore, 
the set of singular points of the fibration $\pi$ in $X$
is a disjoint union of circles transverse to the foliation $\F_X$.

\begin{definition}
  \begin{enumerate}
  \item A {\em symplectic Lefschetz fibration} is a topological Lefschetz fibration $\pi : X \to B$ with a symplectic form $\om \in \Om^2(X)$ that is non-degenerate on the fiber tangent spaces $\ker(d \pi_x)$ for non-singular points $x \in X$, or
  \item  A {\em foliated symplectic Lefschetz fibration} is a topological  Lefschetz fibration $\pi : (X,\F_X) \to (B,\F_B)$ on a foliated base with a strong symplectic form $\om \in \Om^2(X)$ that is non-degenerate on the fiber tangent spaces $\ker(d \pi_x)$ for non-singular points $x \in X$. We drop the word ``foliated'' from the terminology whenever it is clear from the context that $X$ and $B$ are foliated manifolds.
  \end{enumerate}
\end{definition}

\begin{proposition}
  \label{prop:symplef}
  Suppose $(B,\F_B)$ is a foliated three-manifold with a strong
  symplectic form $\om_B$, with boundary $\partial B$ tangent to the
  foliation.  Suppose $\pi: (X,\F_X) \to (B, \F_B)$ is a topological
  Lefschetz fibration that has a section $s : B \to X$ whose image
  does not contain singular points of the fibration.  Then, for any
  $\lam>0$, there is a strong symplectic form $\om$ on $(X,\F_X)$
  which makes $(X,\F_X)$ a symplectic Lefschetz fibration, and for
  which the regular fibers of $\pi$ have area $\lam$.
\end{proposition}

\begin{remark}
  The hypothesis of Proposition \ref{prop:symplef} requires the existence of a
  section on the fibration.  If the regular fiber is an oriented surface of
  genus at least $1$, this condition is equivalent to requiring that
  the homology class $[F_b]$, $b \in B$ of a regular fiber is non-zero
  in $H_2(X,\R)$, which is the hypothesis in \cite[Theorem
  10.2.18]{Gompf_book}.  In fact, if the genus of $F_b$ is at least
  two, the hypothesis (and its equivalent version) is true for any $X$. In the case of genus $1$,
  the hypothesis of Proposition \ref{prop:symplef} may fail to hold, such as in
  \cite[Example 6.1.6]{ms:small}, and a symplectic form does not exist as a result.
\end{remark}

\begin{proof}[Proof of Proposition \ref{prop:symplef}] 
  The first step is to construct a closed two-form $\zeta$ on $X$,
  whose restriction to fibers integrates to $\lam$.  The image
  $B_s:=s(B)$ of the section $s$ has an oriented normal bundle $NB_s$.
  Identifying a neighborhood $U_{B_s} \subset X$ of $B_s$ to a neighborhood of the zero section of $NB_s$, the required two-form $\zeta$ is a representative of
  the Thom class of $NB_s$ supported in $U_{B_s}$ (\cite[Section 6]{Bott-Tu})
  extended by zero to all of $X$, and
  multiplied by $\lam$.
  We also
  fix a symplectic form $\om_F \in \Om^2(F)$ that integrates to $\lam$ on
  the fiber $F$.


  For any point $b \in B$, we will now define a two-form $\om_b$ on
  $\pi^{-1}(U_b)$ for a neighborhood $U_b \subset B$ of $b$, that is
  symplectic on the regular fibers of $\pi$ and has area $\lam$ on
  each fiber.  If $b$ is a regular value of $\pi$, we can choose an
  orientation-preserving 
  trivialization $\pi^{-1}(U_b) \simeq U_b \times F$ and define
  $\om_b \in \Om^2(\pi^{-1}(U_b))$ to be the pullback of the symplectic
  form $\om_F \in \Om^2(F)$.  If $b \in B$ is a singular value of the
  fibration, we take a foliated chart
  $U_b \simeq V_b \times (-\eps,\eps)$ on $B$ such that $b \in V_b \times \{0\}$ 
  and $\{b\} \times (-\eps,\eps)$ is the set of singular values in
  $U_b$.  A symplectic form is constructed on $\pi^{-1}(V_b)$ by Lemma
  \ref{lem:singsymp}, and $\om_b$ is defined to be its pullback to
  $\pi^{-1}(U_b)$.

  The global form $\om$ is obtained by patching.
  For any $b \in B$, the fibration $\pi^{-1}(U_b)$ is either trivial or the union of its singular fibers is
  $F_b \times (-\eps,\eps)$, and in both cases $\pi^{-1}(U_b)$ has a deformation retraction to $F_b$.
  The deformation retraction, together with integration on the fiber $F_b$, gives isomorphisms
  $H^2(\pi^{-1}(U_b)) \simeq H^2(F_b)  \simeq \R$.  Since $\om_b -\zeta$
  integrates to zero on $F_b$, we conclude that 
  the 
  difference $\om_b - \zeta$ is exact on $\pi^{-1}(U_b)$. 
  So, $\om_b -\zeta=d\alpha_b$ for some
  $\alpha_b \in \Om^1(\pi^{-1}(U_b))$. The base $B$ can be covered by
  a finite subset of $\{U_b\}_b$, denoted by $\mU$. Let
  $\{\eta_b:B \to [0,1]\}_{b \in \mU}$ be a partition of unity for the
  finite cover.  Define a form
  \begin{equation*}
    \om_{pre}:= \zeta + \sum_{b \in \mU} d(\eta_b \alpha_b).
  \end{equation*}
  This form is symplectic on each fiber $F_{b_0}$ and has the right
  volume, because on $F_{b_0}$, the form
  $\om_{pre}=\sum_{b \in \mU} \eta_b(b_0)\om_b$ is a convex
  combination of the forms $\om_b$.  Finally, for a large $C>0$,
  $\om:=\om_{pre} + C\pi^*\om_B$ is a symplectic form on the leaves of
  $X$, finishing the proof of the proposition.
\end{proof}

\section{The example}\label{sec3}
In this section, we construct the example that proves Theorem
\ref{thm:main}. The following proposition constructs a strong
symplectic foliation that has an embedded loop in a leaf that is
non-contractible in the leaf but contractible in the manifold (as in
Theorem \ref{thm:main} \eqref{part1}). Later in the section, we prove
Theorem \ref{thm:main} by modifying the foliation so that it has a
transversal loop that is contractible (as in Theorem \ref{thm:main}
\eqref{part2}). Given a Riemann surface $F$ of genus $g \geq 1$, and
an embedded loop $V \subset F$, the proposition constructs a Lefschetz
fibration whose regular fiber is $F$ and $V$ is a vanishing cycle.
The embedded loop $V$ is {\em essential}, by which we mean there is a
set of $2g$ generators of $\pi_1(F)$ that contains $[V]$.

\begin{proposition} \label{prop:lef} Let $F$ be an oriented surface of
  genus $g \geq 1$ and let $V \subset F$ be an essential embedded loop.  There
  is a foliated symplectic Lefschetz fibration $\pi: (X,\F_X) \to (B, \F_B)$
  with regular fiber $F$ satisfying the following: There is a
  compact leaf $L_1$ of $\F_X$ on which the fibration is trivial, that is, $L_1=\bT^2 \times F$ and the loop $\{point\} \times V \subset L_1$ is contractible in $X$. 

  %
  %
\end{proposition}

\begin{figure}[h]
  \centering \scalebox{1}{ 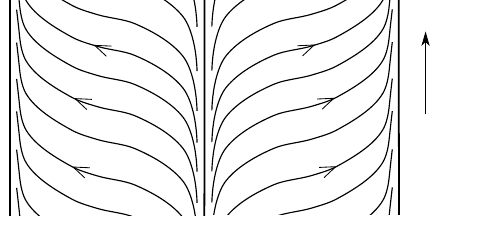}
  \caption{Almost horizontal foliation on a strip $[-1,1] \times \R$.}
  \label{fig:ah}
\end{figure}
\begin{proof}
  We use an {\em almost horizontal foliation} on the base manifold,
  which is introduced in \cite[page 626]{El:ot}, and which we describe
  here. (See also Remark \ref{rem:ah} below.)  The base manifold is a
  $3$-dimensional torus denoted by
  $B=\bS^1_x \times \bS^1_y \times \bS^1_z$, where $\bS^1_x:=\R/2\Z$ and $\bS^1_y, \bS^1_z:=\R/\Z$ are circles. 
  Let $x$, $y$, $z$ denote
  the projection maps on $B$ to the respective components. 
  The foliation $\F_B$ is the pullback of
  a foliation $\F_{\bT^2}$ on $\bS^1_x \times \bS^1_y$, which in turn
  is the quotient of a foliation $\F_{strip}$ on the strip
  $[-1,1] \times \R$.  The foliation on the strip $\F_{strip}$ has the
  lines $\{x=\pm 1, 0\}$ as closed leaves, and the non-closed leaves
  are $\{y=c+\on{sign}(x)\tan (\pi x-\frac \pi 2)\}_{c \in \R}$, see Figure
  \ref{fig:ah}. This foliation is invariant under translation in the
  $\R$-direction, and so, descends to a foliation $\F_{\bT^2}$ on the
  quotient $([-1,1] \times \R/\Z)/(-1,y) \sim (1,y)$. Finally, $\F_B$
  is defined by pullback.  The form $\om_B:=dy \wedge dz$ is a strong
  symplectic form on the base $(B,\F_B)$, and it induces an orientation on
  the leaves of $\F_B$. For future use, for any interval
  $I \subset [-1,1]$, we denote $B_I:=\{b \in B: x(b) \in I\}$.

  We will describe a topological Lefschetz fibration $\pi:X \to B$ on the foliated base $B$
  with regular fiber $F$
  satisfying the following properties:
  \begin{enumerate}
  \item \label{cond1}
    {\rm(Singular values)}
    The set of singular values of the fibration $\pi$ consists of two
    loops $P_+ \subset B_{(0,1)}$, $P_- \subset B_{(-1,0)}$ transverse to the
    foliation, $\F_B$ defined as 
  \begin{equation}
    \label{eq:ppm}
    P_\pm:=\{(\pm \thh, y ,z_0) : y \in \bS^1\}  
  \end{equation}
  for some fixed $z_0 \in \bS^1_z$.  As a result, the compact leaves in
  $X$ do not have singular points, and each non-compact leaf has one
  singular point.
\item \label{cond2}
  {\rm(Compact leaf $L_1$ with trivial fibration)}
  The fibration is trivial on the compact leaf
    $L_1=\{x= \pm 1\} $. That is, $L_1=\bT^2 \times F$.
  \item \label{cond0} {\rm(Vanishing cycle)}
There is 
    a family of
    loops $\V:=\{V_b \subset F_b\}_{b \in B \bs (\pi(\Delta))}$
    on regular fibers $F_b$ and a trivialization
    \begin{equation}
      \label{eq:vctriv}
      X \bs (\cup_b V_b \cup \Delta) \simeq B \times (F \bs V), 
    \end{equation}
    which is equal to the restriction of the standard trivialization on $L_1$. 
%
    
  \item \label{cond3}
{\rm(Compact leaf $L_0$ with a non-trivial fibration)}
    On the leaf $L_0:=\{x=0\}$, the fibration has trivial
  monodromy along the loop $\bS^1_y$, and along the loop $\bS^1_z$ the
  monodromy map is a $-1$ Dehn twist along the loop
  $V \subset F$.
  %
  %
  \end{enumerate}
  We justify the existence of a fibration satisfying these properties:
  To describe the fibration, we focus on the region
  $B_{[0,1]} \subset B$, since the construction on $B_{[-1,0]}$ is
  symmetric.  The topological Lefschetz fibration $X|\{x=0, 1\}$ on
  the compact leaves is determined by conditions \eqref{cond2} and
  \eqref{cond3}. We fix the fibration and the identification
  \eqref{eq:vctriv} to $B_{\{0,1\}} \times (F \bs V)$.  The fibration
  extends to a neighborhood $X|\{x \leq \eps, x \geq 1-\eps\}$ for a
  small $\eps>0$.  Therefore, on a non-compact leaf
  $L \simeq \R \times S^1$ in $B_{(0,1)}$ , we have a trivial
  fibration in a neighborhood of $+\infty$ and a fibration with a $(-1)$-Dehn
  twist near $-\infty$. The fibration on the ends of $L$ extends to a fibration $X_L \to L$ with a single singular point in the fiber over
  $P_+ \cap L$, see Figure \ref{fig:hol} (left). 
  A similar operation can be done in an $\bS^1$-family, and 
  the fibration $X|\{x \leq \eps, x \geq 1-\eps\}$ extends to all of $B_{[0,1]}$ with $P_+$ being the set of singular values. 
  Indeed, $X|B_{(0,1)}$ is homeomorphic to the product foliation
  $\bS^1 \times X_L$.
\footnote{By definition the leaves of the product foliation $\bS^1 \times X_L$ are $\{point\} \times X_L$.} 
  There is no obstruction in extending the family of vanishing cycles $\{V_b\}_b$ and the map \eqref{eq:vctriv} to $X|B_{[0,1]}$.
  The description of the
  topological Lefschetz fibration $X|B_{(-1,0)}$ is analogous.
  

  \begin{figure}[h]
    \centering \scalebox{1}{ 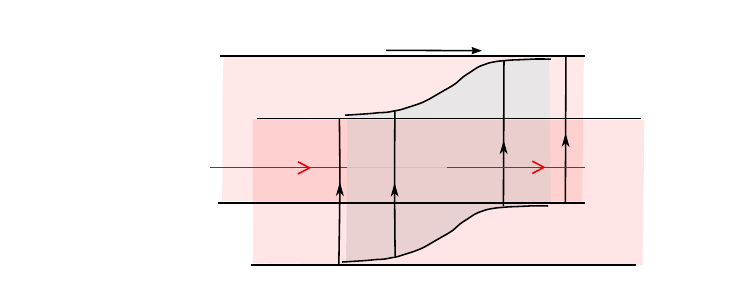}
    \caption{{\sc Left:} A non-compact leaf $L$ in the base $B$. The
      monodromy around the singular point is equal to the difference
      in the monodromies around $\beta_u$ and $\beta_l$. {\sc Right:}
      As the loop $\beta_u$ resp. $\beta_l$ is moved to the end of the
      cylinder, it limits to a loop $\beta_u'$ resp. $\beta_l'$ in the
      compact leaf $x=1$ resp. $x=0$. The monodromy map is trivial on
      $\beta_u$, $\beta_u'$. On the loops $\beta_l$, $\beta_l'$, the
      monodromy map is a negative Dehn twist about $V$.}
    \label{fig:hol}
  \end{figure}

  The topological Lefschetz fibration fibration $X \to B$ has a global section because of the trivialization \eqref{eq:vctriv}. Therefore, by Proposition \ref{prop:symplef}, the space $X$ can be given a
  strong symplectic form making it a symplectic Lefschetz fibration.

  The map $\pi_1(\F) \to \pi_1(X)$ is not injective. Indeed, consider
  a loop $\gamma$ in the compact leaf $\{x= 1\}$ that lies in a fiber
  of $\pi$ and is homotopic to $V$. This loop is non-contractible in
  the compact leaf, but becomes contractible when it is displaced to a
  nearby non-compact leaf since it is a vanishing cycle of the
  Lefschetz fibration in the non compact leaf.
\end{proof}

To finish the proof of Theorem \ref{thm:main}, we need a notion of
holonomy transport of a transversal along a loop in a leaf.  A {\em
  transversal} $\tau$ is a path in $X$ that is transverse to the $\F$.
If $\gamma \subset L$ is a short enough path on a leaf $L$ of a
foliation $\F$, and $\tau_0$, $\tau_1$ are transversals passing through 
the end-points of $\gamma$, then the product structure of a foliated
chart determines a homeomorphism $h : \tau_0 \to \tau_1$ from a germ
of $\tau_0$ to a germ of $\tau_1$, called the {\em holonomy transport}
along $\gamma$. The map $h$ is invariant under homotopies of $\gamma$
relative to end points.  Concatenating paths on a leaf, we obtain a
homomorphism
\begin{equation}
  \label{eq:hol}
  h : \pi_1(\F,x_0) \to \on{Homeo}(\tau),
\end{equation}
where $\pi_1(\F,x_0)$ is the homotopy class of loops contained in the
leaf $L$ with base point $x_0 \in L$, and $\on{Homeo}(\tau)$ is the
group of germs of self-homeomorphisms of $\tau$ that fix the point $L \cap \tau$.
See \cite[Theorem 4.4]{Cal}
for details.

\begin{proof}
  [Proof of Theorem \ref{thm:main}] By Proposition \ref{prop:lef},
  there is a symplectic Lefschetz fibration
  \[\pi:(X,\F_X) \to (B, \F_B)\]
  with strong symplectic form $\om \in \Om^2(X)$ that satisfies condition 
  \eqref{part1} in Theorem \ref{thm:main}.  That is, there is a loop
  $\gamma$ in the leaf $L_1:=\{x=1\}$ of $X$ that is non-contractible
  in $L_1$ but contractible in $X$.
  
  To produce a contractible transversal required by \eqref{part2} in 
  Theorem \ref{thm:main}, we perturb the foliation $\F$ so that
  there is a transversal loop homotopic to $\gamma$ as follows.  We recall that
  the leaf $L_1$ is the product $\bT^2 \times F$ where $\bT^2$ is a
  leaf in the base $(B,\F_B)$ and $\gamma$ is an essential loop in the
  fiber $\{b\} \times F$, which is an oriented surface of genus
  $g \geq 1$.
 We
  replace the single compact leaf $L_1$ by a family of compact leaves
  $L_1 \times [-\eps,\eps]$ via a $C^0$-small perturbation of the
  foliation, and denote the leaves by $L_1^t:=L_1 \times \{t\}$.
  Next, we apply Lemma \ref{lem2} (stated below) 
  to the leaf $L_1^0$ with the loop $\gamma$ and the hypersurface
  $W:=\bT^2 \times \gamma'$, where 
  $\gamma' \subset F$ is an embedded loop that
  intersects $\gamma \subset F$ transversely once.
  The loop $\gamma$ and the hypersurface $W$ have trivial holonomy in the leaf $L_1^0$, because the foliation is trivial in a neighborhood of $L_1^0$. 
  Lemma \ref{lem2} produces a $C^0$-small perturbation $\F'$ of $\F$ which has a transversal homotopic to $\gamma$.
  Finally we point out that if a perturbation of the foliation is $C^0$-small, $\om$ continues to be a strong symplectic form. 
  This finishes the proof of Theorem \ref{thm:main}.
\end{proof}

The following lemma was used in the proof of Theorem \ref{thm:main}. It describes a surgery operation on a foliation that produces a transverse loop
homotopic to a loop in a leaf.
\begin{lemma}\label{lem2}
  {\rm(A cut and shear surgery)} Suppose $(X,\F)$ is a foliated
  manifold with a leaf $L$ containing
  \begin{enumerate}
  \item an embedded loop $\gamma: \bS^1 \to L$
    with trivial holonomy,
  \item and a compact embedded hypersurface $W \subset L$ that
    intersects $\gamma$ transversely at a single point, and such that the
    restriction of the holonomy of $L$ to $W$ is trivial.
  \end{enumerate}
  Then, there exists a foliation $\F'$ that is $C^0$-close to $\F$ for
  which $\gamma$ is still in a leaf of $\F'$ and $\gamma$ has
  non-trivial holonomy in $\F'$. Furthermore, there is a loop
  $\gamma' : \bS^1 \to X$ that is homotopic to $\gamma$ and transverse
  to $\F'$.
\end{lemma}
\begin{proof}
  Since the holonomy of the leaf $L$ is trivial over $W$, there is an
  embedding of $W \times (-\eps,\eps)$ into $X$ such that for any $w$,
  $(w,0)$ maps to $w$,  and for each $t$, $W \times \{t\}$ is
  mapped to a leaf.

  We {\em cut and shear} along $W \times (-\eps,\eps)$ as in \cite[Example 4.16]{Cal}.
  That is, we cut the manifold $X$ along $W \times (-\eps,\eps)$ and glue by a
  diffeomorphism
  \begin{equation*}
    \Phi:=(\Id,\phi) :  W  \times (-\eps,\eps)\to W \times (-\eps,\eps)
  \end{equation*}
  where $\phi:(-\eps,\eps) \to (-\eps,\eps)$ is a diffeomorphism that is identity near
  $\pm \eps$, $\phi(0)=0$ and $\phi'(0)>1$. Let $(X',\F')$ be the
  resulting foliated manifold. There is a homeomorphism $i:X \to X'$
  that is identity away from a small neighborhood of
  $W \times (-\eps,\eps)$, and the pullback foliation $i^*\F'$ is
  smooth.  The $C^0$-distance between $\F$ and $\F'$ is controlled by
  $\Mod{\phi-\Id}_{C^0}$ and can be made arbitrarily small.  The construction implies that 
  the holonomy of the loop $\gamma$ in $\F'$ is
  $\phi \in \on{Homeo}(-\eps,\eps)$, see also \cite[Section 2.1.2]{El:ot}. See Figure \ref{fig:move}.

  In the new foliation $\F'$, the loop $\gamma$ can be deformed to a transversal as follows.
  Since $\gamma$ has trivial holonomy in $(X,\F)$, it extends to an embedding 
  $\tilde \gamma : \bS^1 \times (-\eps,\eps) \to X$ so that
  $\tilde \gamma(\cdot,0) \equiv \gamma$ and 
  $\tilde \gamma( \bS^1 \times \{t\})$ lies in
  the leaf of $\F$ containing $W \times \{t\}$.
  Choose $0<\delta <\eps$ so that $\phi(\delta)>\delta$. The path
  \[\gamma': [0,1] \to X, \quad 
    t \mapsto \tilde \gamma(t,\delta + (\phi(\delta)-\delta)t)\]
  is transverse to $\F$ and its end-points are identified to each other via the cut and shear operation. Therefore, $\gamma'$ is a loop in $(X',\F')$ that is transverse to $\F'$. 
\end{proof}

\begin{figure}[h]
  \centering \scalebox{.8}{ 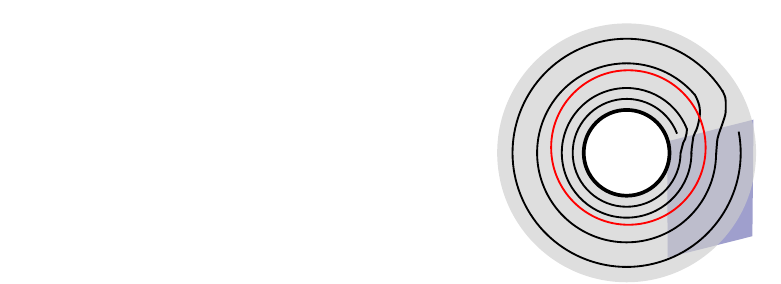}
  \caption{ Cut and shear shown in a plane transverse to $W$. In $(X,\F')$, $\gamma'$ is a transversal.}
  \label{fig:move}
\end{figure}
  
\begin{remark}\label{rem:ah}
  {\rm(Almost horizontal foliation by cut and shear)}
  The almost horizontal foliation on the annulus defined in
  Proposition \ref{prop:lef} is diffeomorphic to the following
  foliation produced by the cut and shear surgery of Lemma \ref{lem2}. We start with the
  {\em horizontal} foliation on $[-1,1] \times \bS^1_y$ whose leaves
  are $\{x=t\}$, $t \in [-1,1]$. The surgery is performed by cutting
  along $\{y=y_0\}$ for some $y_0 \in \bS^1_y$ and gluing by the
  identification $(x,y_0^-) \sim (\phi(x), y_0^+)$ where
  $\phi : [-1,1] \to [-1,1]$ is a diffeomorphism satisfying
  \[\phi(\pm 1)=1, \phi(0)=0, \quad \phi(x)  < x \text{ on $(-1,0)$,} \quad \phi(x)>x \text{ on $(0,1)$.} \]
\end{remark}

\bibliographystyle{amsplain} \bibliography{novbiblio}

\end{document}